\documentclass[a4paper,reqno]{amsart}
\pdfoutput=1

\usepackage{xcolor}
\usepackage{enumerate} 
\usepackage{enumitem}
\usepackage{dsfont}
\usepackage{relsize}
\usepackage{amsmath,amsfonts,amssymb,amsthm,
stmaryrd,bbm,graphicx,mathtools,enumerate}

\usepackage{mathrsfs}
\usepackage[T1]{fontenc} 
\usepackage[latin1]{inputenc}
\usepackage[english]{babel}
\usepackage[tracking,spacing,kerning,babel]{microtype}
\usepackage{wasysym} 
\usepackage{color}
\usepackage{xcolor}
    	\usepackage{framed} 
\usepackage{wrapfig} 

\usepackage[final]{hyperref}   
\hypersetup{
    linktoc=page,
    linkcolor=red,          
    citecolor=blue,        
    filecolor=blue,      
    urlcolor=cyan,
   colorlinks=true           
}

\usepackage{mathalfa}

\usepackage{upgreek}

\usepackage{lipsum} 

\usepackage{minitoc}
\usepackage{multicol}
\usepackage[mathscr]{}

\usepackage{moresize}

\usepackage{enumerate} 
\usepackage{enumitem}

\usepackage{amsmath,amsfonts,amssymb,amsthm,stmaryrd,bbm,graphicx,mathtools,enumerate}
\usepackage{mathrsfs}
\usepackage[T1]{fontenc} 
\usepackage[latin1]{inputenc}
\usepackage[english]{babel}
\usepackage[tracking,spacing,kerning,babel]{microtype}
\usepackage{wasysym} 
\usepackage{color}
\usepackage{xcolor}
    	\usepackage{framed} 
\usepackage{wrapfig} 

\usepackage[final]{hyperref}   
\hypersetup{
    linktoc=page,
    linkcolor=red,          
    citecolor=blue,        
    filecolor=blue,      
    urlcolor=cyan,
   colorlinks=true           
}

\usepackage{lipsum} 

\usepackage{minitoc}
\usepackage{multicol}

\newtheoremstyle{mine}
{\baselineskip}
{\baselineskip}
{\itshape}
{
}
{\bfseries}
{.}
{.5em}
{#1 #2\ifx#3\relax\else~(#3)\fi}

\theoremstyle{mine}

\newtheorem{theorem}{Theorem}
\numberwithin{theorem}{section}
\newtheorem{corollary}[theorem]{Corollary}
\newtheorem{proposition}[theorem]{Proposition}
\newtheorem{lemma}[theorem]{Lemma}

\newtheorem{definition}[theorem]{Definition}

\numberwithin{equation}{section}

\theoremstyle{remark}
\newtheorem{remark}{Remark}

\colorlet{shadecolor}{blue!10}

\renewcommand{\epsilon}{\varepsilon}

\newcommand{\R}{\mathbb{R}}

\newcommand{\Q}{\mathbb{Q}}

\def\H{\mathbb{H}}
\def\C{\mathbb{C}}

\def\P{\mathbb{P}} 
\def\E{\mathbb{E}} 

\def \eps {\epsilon}

\def\<#1{\langle #1\rangle}

\def\bi{\begin{itemize}}  
\def\ei{\end{itemize}}
\def\bnum{\begin{enumerate}} 
\def\enum{\end{enumerate}}

\title
[]
{
Large Deviation Principle for complex solution to squared Bessel SDE\MakeLowercase{}
}

\author{Arnab Chowdhury, Atul Shekhar}

\address[Arnab Chowdhury]
{Tata Institute of Fundamental Research-CAM, Bangalore, India}

\email{arnab2020@tifrbng.res.in}

\address[Atul Shekhar]
{Tata Institute of Fundamental Research-CAM, Bangalore, India}
\email{atul@tifrbng.res.in}

\begin{document}

\maketitle

\begin{center}
{\em }
\end{center}

\begin{abstract}
Complex solutions to squared Bessel SDEs appear naturally in relation to Schramm-Loewner evolutions. We prove a large deviation principle for such solutions as the dimension parameter tends to $-\infty$. 

\end{abstract}

\section{Introduction}\label{intro}

\subsection{Context.}
In this article we prove the Large deviation principle (LDP) for the complex solution to squared Bessel SDE. For a precise definition of a LDP and usual notions related to it, we refer to \cite{dembo2009large,deuschel2001large}. For $\delta \geq 0$, the classical $\delta$-dimensional squared Bessel process is the non-negative solution to the squared Bessel SDE 
\begin{equation}\label{RBSDE}
dX_t = 2 \sqrt{X_t}dB_t + \delta dt, \hspace{2mm} X_0=x \geq 0,
\end{equation}
where $B$ is a standard Brownian motion defined on some probability space $(\Omega, \mathcal{F}, \P)$, see [\cite{RY}-ChapterXI]. In relation to Schramm-Loewner-Evolutions (SLEs), it is natural to consider a variant of \eqref{RBSDE} for $\delta <0$ and with complex valued solutions. More precisely, for $\eta >0$ (we write $\eta =-\delta$), we consider the SDE 
\begin{equation}\label{CBSDE}
dY_t = 2 A_t dB_t - \eta dt, \hspace{2mm} Y_0= 0,
\end{equation}
where $Y_t, A_t$ are complex valued adapted processes (w.r.t. the filtration of $B$) such that $A_t^2 = Y_t$ and $Im(A_t) \geq 0$. 
Note that for upper half plane $\H:=\{x+ iy | y>0\}$, the square root function $\sqrt{z}: \C\setminus [0,\infty) \to \H$ is a conformal bijection. As such, if $Y_t \in \C\setminus [0,\infty)$ for some $t$, then $A_t = \sqrt{Y_t}$. Otherwise, if $Y_t \in [0,\infty)$, then $A_t$ makes a choice from $\pm \sqrt{|Y_t|}$ in an adapted way\footnote{If $z\in \C\setminus [0,\infty)$, we write $\sqrt{z}$ for its complex square root so that $\sqrt{z} \in \H$. If $z\in [0,\infty)$, we write $\sqrt{z}$ to mean the standard non-negative square root of $z$. Note that for $x>0$, $\pm\sqrt{x}$ are two possible limit points of $\sqrt{z}$ as $z\to x$ in $\C\setminus [0,\infty)$.}. In other words, $A$ is an adapted branch chosen from all possible square roots of $Y$. We thus refer to $A$ as a branch square root of $Y$. It is proven in \cite{MSY} that if $Y$ is any solution to \eqref{CBSDE}, then almost surely $Y_t \in \C\setminus [0,\infty)$ for all $t>0$. As such, $A_t = \sqrt{Y_t}$ for all $t >0$ and \eqref{CBSDE} is equivalent to 
\begin{equation}\label{CBSDE-2}
dY_t = 2 \sqrt{Y_t}dB_t - \eta dt, \hspace{2mm} Y_0= 0.
\end{equation}
The existence and uniqueness of strong solution to \eqref{CBSDE-2} is a consequence of the Rohde-Schramm estimate \cite{RS05}, see [\cite{MSY}-Theorem $1.5$]. In this article we prove the LDP for solutions $Y^\eta$ as $\eta \to \infty$. The corresponding LDP result for $X^\delta$ as $\delta \to \infty$ was proven in \cite{YorLDP}.\\

%

\subsection{Main result.}
Similarly as for $X^\delta$ in \cite{YorLDP}, we translate the LDP for $Y^\eta$ into a small noise LDP as follows. Set $\epsilon = 1/\sqrt{\eta}$ and $Z_t^\epsilon = Y_t^\eta/\eta$. Then, $Z^\eps$ solves 
\begin{equation}\label{Z-eqn}
dZ_t^\eps = -dt + 2\eps\sqrt{Z_t^\eps}dB_t, \hspace{2mm} Z_0^\eps = 0.
\end{equation}

Let $C_0([0,T], \C)= \{\varphi :[0,T] \to \C \textrm{ }\bigl | \textrm{ } \varphi \textrm{ is continuous and } \varphi_0=0\}$  be equipped with the uniform metric. We view $Z^\eps$ as a $C_0([0,T], \C)$ valued random variable and denote the law of $Z^\eps$ by $\mu^\eps$. Our main result gives the LDP for the family $\{\mu^\eps\}_{\eps >0}$. Let us first describe the LDP rate function $I$ for $\{\mu^\eps\}_{\eps >0}$. The rate function $I$ is finite for functions $\varphi$ which satisfy condition: (H1) $\varphi \in C_0([0,T], \C) $ such that $\varphi_t \in \C\setminus [0,\infty)$ for all $t>0$, (H2) $\varphi_t$ is absolutely continuous, i.e. both $Re(\varphi_t)$ and $Im(\varphi_t)$ are absolutely continuous, and (H3) $(\dot{\varphi}_t + 1)/(2\sqrt{\varphi_t})$ is real valued with $(\dot{\varphi}_t + 1)/(2\sqrt{\varphi_t}) \in  L^2([0,T], \R)$. Let $\mathcal{D}([0,T], \C) := \{ \varphi \textrm{ }\bigl | \textrm{ } \varphi \textrm{ satisfies H1, H2, H3} \}$.

\begin{theorem}\label{main-thm}
The family $\{\mu^\eps\}_{\eps >0}$ satisfies the LDP with speed $\epsilon^2$ and a good rate function $I(\cdot)$ defined by 
\begin{equation}\label{def-rate}
I(\varphi)=
    \begin{cases}
         & \int_0^T \frac{(\dot{\varphi}_t + 1)^2}{8\varphi_t}dt \textrm{ if } \varphi \in \mathcal{D}([0,T], \C), \\
        &+\infty  \textrm{ otherwise. }
    \end{cases}
\end{equation}

\end{theorem}

\subsection{Motivation.}
 The process $Y^\eta$ is related to SLE$_\kappa$ as follows: for $\kappa <4$ and $\eta = 4/\kappa -1$, 

\begin{equation}\label{SLE=flow}
\{\sqrt{\kappa Y^\eta(T-t, T,0)}\}_{t\in [0,T]} \overset{d}{=} \{\gamma_t\}_{t\in [0,T]},
\end{equation}
 where $\gamma_t$ is the SLE$_\kappa$ curve and $\{Y^\eta(s,t,z)\}_{0\leq s\leq t, z \in \overline{\H}}$ is the flow of solutions obtained by solving \eqref{CBSDE-2} with the initial condition $Y_s =z \in \overline{\H}$, see [\cite{MSY}-Corollary $1.7$] for details. Our motivation to prove a LDP for $Y^\eta$ comes from the work of Y. Wang \cite{yilin-1,yilin-2,yilin-3} on the LDP for SLE$_\kappa$ as $\kappa \to 0+$. A LDP result for SLE$_\kappa$ with respect to Hausdorff metric was proven in \cite{wang-peltola}. To establish a LDP for SLE$_\kappa$ in the (stronger) uniform metric, one can utilise \eqref{SLE=flow} and reduce this problem to proving a LDP for the stochastic flow of \eqref{CBSDE-2}, see \cite{stroock-LDP-flow, ben-arous-castell-flow-LDP} for some results in that direction. A natural first step in this approach is to prove a LDP for the solution $Y^\eta$ itself which is addressed in this paper (note that $\eta \to \infty$ as $\kappa \to 0+$). We note that a LDP result for SLE$_\kappa$ with respect to uniform metric (but in an incomplete space $(S,\tau))$ has been established by V. Guskov \cite{guskov}. However, the LDP for $Y^\eta$ does not follow from results of \cite{guskov}. \\
As a corollary of Theorem \ref{main-thm}, one can obtain a large deviation estimate for the tip of SLE$_\kappa$. Using \eqref{SLE=flow}, the tip $\gamma_T^\kappa$ of SLE$_\kappa$ is given by $\sqrt{\kappa Y_T^\eta}$, where $\eta = 4/\kappa -1$. Theorem \ref{main-thm} can hence be applied to obtain a LDP for $\gamma_T^\kappa$. This can be compared to results of \cite{lawSLE-TIP} which describes the exact law of the tip $\gamma_T^\kappa$. The corresponding rate function in the LDP for $\gamma_T^\kappa$ is given by $I(z) = \inf \{ I(\varphi) \textrm{ }| \textrm{ }\varphi \textrm{ joins $0$ to $z^2$}\}$, where $z\in \H$ and $I(\varphi)$ is given by \eqref{def-rate}. In the language of \cite{yilin-1}, this is the minimum Loewner energy required for a curve to pass through $z$. This can be explicitly computed and it turns out to be $-8\log(\sin(\arg(z)))$. This was already computed in \cite{yilin-1} using probabilistic methods. A more direct deterministic proof has been given by T. Mesikepp \cite{mesikepp}. A yet another proof of this fact can also be obtained by using Euler-Lagrange equation to directly compute the minimum value $I(z) = \inf \{ I(\varphi) \textrm{ }| \textrm{ }\varphi \textrm{ joins $0$ to $z^2$}\}$ from \eqref{def-rate}. Since this computation is long and not the main point of this paper, we do not present the details here. \\ 
Theorem \ref{main-thm} is also a natural variant of LDP for squared Bessel processes $X^\delta$ as proven in \cite{YorLDP}. It follows from central limit theorem and the additive property of squared Bessel processes (cf.  [\cite{YorLDP}, Equation $(1.3)$]) that as $\delta \to \infty$, 
\begin{equation}\label{BCLT}
\biggl\{\sqrt{\delta}\biggl(\frac{X_t^\delta}{\delta} -t \biggr) \biggr\}_{0\leq t\leq T}\overset{d}{\longrightarrow} \bigl\{\sqrt{2}B_{t^2}\bigr\}_{0\leq t\leq T }.
\end{equation}

In our setting, processes $Y^\eta$ does not satisfy the additive property. Nevertheless, we can write 
\[\sqrt{\eta}\biggl(\frac{Y_t^\eta}{\eta} + t \biggr) = \frac{1}{\epsilon}\left(Z^\epsilon_t+t\right) = 2\int_0^t\sqrt{Z_r^\eps}dB_r. \]
It can be easily verified that $\{Z_t^\eps\}_{t\in [0,T]} \to \{-t\}_{t \in [0,T]}$ in $L^2(\P)$ as $\eps \to 0$. Hence, it follows that as $\eta \to \infty$, 
\begin{equation}\label{CBCLT}
\biggl\{\sqrt{\eta}\biggl(\frac{Y_t^\eta}{\eta} + t \biggr) \biggr\}_{0\leq t\leq T}\overset{d}{\longrightarrow} \biggl\{2i\int_0^t\sqrt{r}dB_r\biggr\}_{0\leq t\leq T} \overset{d}{=}\bigl\{\sqrt{2}iB_{t^2}\bigr\}_{0\leq t\leq T}.
\end{equation}
Hence, even though $Y^\eta$ does not satisfy the additivity property, we do have the above variant of CLT for $Y^\eta$. Obtaining a LDP for $Y^\eta$ is a natural next step.

\subsection{Idea of the proof of Theorem \ref{main-thm}.}

To prove Theorem \ref{main-thm}, we use the standard argument based on exponential martingales. We first show that the family $\{\mu^\eps\}_{\eps >0}$ is exponentially tight which is an easy consequence of estimates in \cite{RTZ} (Lemma \ref{U-V-bound} below). Then, we prove the weak upper and the weak lower bound ( Proposition \ref{UP} and Proposition \ref{LB} below). The weak upper bound is obtained by weighting $\P$ with the exponential martingale $M_{f,g}^\eps(Z^\eps)$ (see \eqref{our-choice} below). The choice of this appropriate martingale $M_{f,g}^\eps(Z^\eps)$ is a key observation of this paper. The weak upper bound is completed by obtaining a variational description of the rate function $I$ which in itself is a two dimensional functional optimisation problem, see Proposition \ref{J=I}. For the weak lower bound, we use the classical change of measure appearing in Cameron-Martin theorem, see \eqref{our-choice-2}. The weak lower bound \eqref{wLB} boils down to Proposition \ref{req-conv} which is another key input of this paper, see Remark \ref{good-remark}. The proof of Proposition \ref{req-conv} relies crucially on results of \cite{STW-1},  particularly on the uniqueness of solution to \eqref{h-phi-eqn}, see Lemma \ref{unique} below. The paper \cite{STW-1} is a foundation for this paper and its results are used repeatedly in several instances.

\begin{remark}

As shown in \cite{YorLDP}, the LDP for $X^\delta$ can be obtained via two other methods besides the approach using exponential martingales: $(1)$ by using an infinite dimensional Cramer's theorem approach which is based on additivity property of $X^\delta$, and $(2)$ by using contraction principle applied to Bessel processes $\sqrt{X^\delta}$, which in turn is based on the work of McKean \cite{mckean-1,mckean-2} giving the continuity of the associated It\^o map. However, these two approaches fail to apply to $Y^\eta$. The process $Y^\eta$ does not satisfy the additivity property. Also, the technique of \cite{mckean-1,mckean-2} does not apply to $\sqrt{Y^\eta}$ and the associated It\^o map is not well defined.  

\end{remark}

\subsection*{Organization of the paper.} 
In section \ref{prelim} we recall some known results which will be useful in the proof of Theorem \ref{main-thm}. Section \ref{section-thm-proof} contains the proof of Theorem \ref{main-thm}.

\subsection*{Acknowledgements.}
A.S. wishes to thank Yilin Wang for various enlightening discussions on Large deviation for SLE, and Bangalore probability seminar group for discussing an earlier draft of this paper.  A.S. and A.C. acknowledges the support from the project PIC RTI4001: Mathematics, Theoretical Sciences and Science Education. 

\section{Preliminaries}\label{prelim}

In this section we recall some results which will be used in the proof of Theorem \ref{main-thm}.

\subsection{Cameron-Martin Perturbations.}
Let $H_0^1([0,T], \R)= \{h: [0,T] \to \R \textrm{ }\bigl | \textrm{ } h_0=0, \dot{h} \in L^2([0,T], \R) \}$ be the Cameron-Martin space equipped with the norm 
\[ ||h||_{H_0^1} = \sqrt{\int_0^T \dot{h}_r^2dr}. \]
For $h \in H_0^1([0,T], \R)$, we will need to consider $Z_t^{\eps, h}$ which are solutions to 

\begin{equation}\label{Z-h-eqn}
dZ_t^{\eps,h} = -dt + 2\sqrt{Z_t^{\eps,h}}( \eps dB_t + dh_t), \hspace{2mm} Z_0^{\eps,h} = 0.
\end{equation}
Using Girsanov theorem, $Z_t^{\eps, h}$ has same almost sure properties as $Z_t^{\eps}$. The existence and uniqueness of strong solution $Z_t^{\eps, h}$ to \eqref{Z-h-eqn} follows similarly as for \eqref{Z-eqn}. We will write $\sqrt{Z_t^{\eps,h}} = U_t^{\eps,h} + iV_t^{\eps,h}$ and $\sqrt{Z_t^\eps} = U_t^\eps + iV_t^\eps$.

\begin{lemma}[Lemma $2.1$ in \cite{RTZ}]\label{U-V-bound}
For $U_t^{\eps,h}, V_t^{\eps,h}$ as above, we have 
\[|U_t^{\eps,h}| \leq 2\sup_{s\in [0,t]}(\eps|B_s| + |h_s|),\]
and
\[V_t^{\eps,h} \leq \sqrt{(\eps^2 + 1)t}.\]
\end{lemma}

We will also need to consider solutions $\varphi^h$ which solves 
\begin{equation}\label{h-phi-eqn}
d\varphi^h_t = -dt + 2 A_t dh_t, \hspace{2mm}\varphi^h_0 =0,
\end{equation}
where $A_t$ is $\overline{\H}$-valued measurable function such that $A_t^2 = \varphi_t^h$, i.e. $A_t = A_t(\varphi^h)$ is a branch square root of $\varphi^h$ similarly as described in Section \ref{intro}. Following results from \cite{STW-1} are crucial inputs in the proof of Theorem \ref{main-thm}.

\begin{lemma}[Proposition $2.6$ in \cite{STW-1}]\label{unique}
Let $h \in H_0^1([0,T], \R)$. For any solution $(\varphi^h, A(\varphi^h))$ of \eqref{h-phi-eqn}, $A_t \in \C \setminus [0,\infty)$ for all $t>0$. Hence, $A_t = \sqrt{\varphi_t^h}$ and \eqref{h-phi-eqn} is equivalent to 
\begin{equation}\label{h-phi-eqn-2}
d\varphi^h_t = -dt + 2  \sqrt{\varphi_t^h} dh_t, \hspace{2mm}\varphi^h_0 =0.
\end{equation}
Furthermore, 
\[ \liminf_{t\to 0+} \frac{Im(\sqrt{\varphi_t^h})}{\sqrt{t}} >0,\]

and \eqref{h-phi-eqn-2} has a unique solution $\varphi^h$\footnote{We in fact only need that $h$ is bounded variation and it satisfies a slowpoint condition, see \cite{STW-1}. It can be easily checked that Cameron-Martin functions $h$ satisfy this slowpoint condition.}.

\end{lemma}

\begin{lemma}[Proposition $3.1$ in \cite{STW-1}] \label{cont-BV}
Let $h_n, h \in H_0^1([0,T], \R)$ such that $h_n \to h$ uniformly as $n\to \infty$. Further assume that $\sup_n ||h_n||_{H_0^1} <\infty$. Then, $\varphi^{h_n}$ converges to $\varphi^h$ uniformly. 
\end{lemma}

\begin{lemma}[Lemma $2.4$ in \cite{STW-1}]\label{bsr-exist}
Let $\varphi^n, \varphi \in C_0([0,T], \C)$ such that $\varphi^n \to \varphi$ uniformly. Suppose for all $n$ and $t>0$, $\varphi_t^n \in \C\setminus [0,\infty)$. Then, there exists a subsequence $
\varphi^{n_k}$ and a branch square root $A= A(\varphi)$ of $\varphi$ such that $\sqrt{\varphi^{n_k}}$ converges uniformly to $A$. \\

\end{lemma}


\section{Proof of Theorem \ref{main-thm}} \label{section-thm-proof}

\subsection{Goodness of rate function $I$.}

Recall the rate function $I(\varphi)$ from \eqref{def-rate}. Note that $I(\varphi) <\infty$ if and only if $\varphi \in  \mathcal{D}([0,T], \C)$. Hence, for Lebesgue almost every $t$, 
\[\frac{\dot{\varphi}_t + 1}{2\sqrt{\varphi_t}} = \dot{h}_t\]
for some $h\in H_0^1([0,T], \R)$. In other words, $\varphi$ solves \eqref{h-phi-eqn-2}. Using Lemma \ref{unique}, it follows that $\varphi = \varphi^h$. Hence, $I(\varphi) <\infty$ if and only if $\varphi = \varphi^h$ for some $h \in H_0^1([0,T], \R)$. In that case, we have 
\[ I(\varphi) = I(\varphi^h) = \frac{1}{2}\int_0^T \dot{h}_r^2 dr. \]

To show that $I$ is a good rate function, we check that level sets $\{\varphi \bigl | I(\varphi) \leq L\}$ is sequentially compact for all $L\geq 0$. Let $\varphi_n \in $ be a sequence such that $I(\varphi_n) \leq L$. Then, $\varphi_n = \varphi^{h_n}$ for some $h_n \in H_0^1([0,T], \R)$ with $||h_n||_{H_0^1} \leq \sqrt{2L}$. Since $(H_0^1, ||\cdot||_{H_0^1})$ is a Hilbert space, its closed balls are weakly compact. Hence, there exists a subsequence $h_{n_k}$ converging weakly in $H_0^1$ to some $h_\infty \in H_0^1([0,T], \R)$ with $||h_\infty||_{H_0^1} \leq \sqrt{2L}$. Also, since $||h_n||_{H_0^1} \leq \sqrt{2L}$, it follows that $\sup_n ||h_n||_{1/2} <\infty$. By Arzela-Ascoli theorem, possibly along a further subsequence, $h_{n_k}$ converges uniformly to $h_\infty$. Using Lemma \ref{cont-BV}, we obtain that $\varphi^{h_{n_k}}$ converges uniformly to $\varphi^{h_\infty} \in \{\varphi \bigl | I(\varphi) \leq L\}$. This implies that $\{\varphi \bigl | I(\varphi) \leq L\}$ is compact.

\subsection{Exponential tightness. }

We prove that the family $\{\mu^\eps\}$ is tight, and it is exponentially tight as well. More precisely:

\begin{proposition}\label{p2}
For any $\alpha \in (0,1/2)$, 

\[\lim_{R \to +\infty}\limsup_{\epsilon \to 0} \epsilon^2 \log\mathbb P(||Z^\epsilon||_\alpha \ge R) =-\infty.\]
Also, for any $h \in H_0^1([0,T], \R)$,
\[ \lim_{R\to \infty} \sup_{\eps\in (0,1)} \P(||Z^{\eps,h}||_\alpha \geq R) = 0.\]

\end{proposition}
\begin{proof}
Fix $\alpha \in (0, 1/2)$. We write $Z_t^\epsilon=-t+2 (M^\epsilon_t + i N_t^\eps)$, where 
\[ M_t^\eps = \eps \int_0^t Re(\sqrt{Z_r^\eps})dB_r  = \eps \int_0^t U_r^\eps dB_r, \textrm{ and }
N_t^\eps = \eps \int_0^t Im(\sqrt{Z_r^\eps})dB_r  = \eps \int_0^t V_r^\eps dB_r\]
are local martingales. Clearly, it suffices to prove that for $f^\eps = M^\eps, N^\eps$
\begin{equation}\label{MET}
\lim_{R \to +\infty}\limsup_{\epsilon \to 0} \epsilon^2 \log\mathbb P(||f^\epsilon||_\alpha \ge R) =-\infty.
\end{equation}


Using the Garsia-Rumsey-Rodemich (GRR) inequality with $\Psi(x)=e^{c\epsilon^{-2}x}-1$ and $p(x)=x^{\frac{1}{2}}$, where $0 <c < 1/2$ is properly chosen constant, we obtain for $f^\eps = M^\eps$ and $f^\eps = N^\eps$  
\begin{align*}
|f^\epsilon_t-f^\epsilon_s| &\le \frac{8\epsilon^2}{c} \int_0^{|t-s|}\log\left(1+\frac{4K^\eps}{u^2}\right)d\sqrt u \leq \frac{8\epsilon^2}{c} (t-s)^{\frac{1}{2}}\left[\log(T^2 + 4K^\eps)+4\log\frac{e}{\sqrt{t-s}}\right], 
\end{align*}

where 
\[ K^\eps:= \int_0^T \int_0^T \biggl\{\exp\left(c\epsilon^{-2}\frac{|f^\epsilon_t-f^\epsilon_s|}{|t-s|^{1/2}}\right)- 1 \biggr\} dsdt.\]

It follows that  
\begin{equation}\label{K-f-bound} 
||f^\eps||_\alpha \lesssim_T  \eps^2(\log(K^\eps + 1) + 1).
\end{equation}

Hence, to obtain \eqref{MET}, it suffices to verify that 

\begin{equation}\label{goal-last} 
\limsup_{\epsilon \to 0} \epsilon^2 \log \int_0^T \int_0^T \E\biggl[\exp\left(c\epsilon^{-2}\frac{|f^\epsilon_t-f^\epsilon_s|}{|t-s|^{1/2}}\right)\biggr]dsdt <\infty.  
\end{equation}


We now use a exponential martingale inequality: for any continuous local martingale $Y$ with $Y_0=0$,  $\mathbb E(e^{\lambda |Y_t|}) \le 2 [\mathbb E(e^{2\lambda^2[Y]_t})]^{1/2}$. Therefore, using Lemma \ref{U-V-bound}, we have for $f_t^\eps = M_t^\eps$ 
\begin{align*}
\hspace{-7mm}\mathbb E \left[\exp\left(c\epsilon^{-2}\frac{|M^\epsilon_t-M^\epsilon_s|}{|t-s|^{1/2}}\right)\right] &\le 
2\left[\mathbb E \left(\exp{\left(\frac{2c^2\epsilon^{-2}}{(t-s)}\int_s^t (U_r^\epsilon)^2 dr\right)}\right)\right]^{\frac{1}{2}} \leq 2\left[\mathbb E \left(\exp{\left(8c^2\sup_{r\in [0,T]} B_r^2 \right)}\right)\right]^{\frac{1}{2}}.
\end{align*}
The $c$ is chosen small enough so that the right hand side above is finite using Fernique theorem. This implies \eqref{goal-last} for  $f_t^\eps = M_t^\eps$. For  $f_t^\eps = N_t^\eps$, again using Lemma \ref{U-V-bound}, we similarly have 

\begin{align*}
\mathbb E \left[\exp\left(c\epsilon^{-2}\frac{|N^\epsilon_t-N^\epsilon_s|}{|t-s|^{1/2}}\right)\right] &\le 
2\left[\mathbb E \left(\exp{\left(\frac{2c^2\epsilon^{-2}}{(t-s)}\int_s^t (V_r^\epsilon)^2 dr\right)}\right)\right]^{\frac{1}{2}} \leq 2 \exp{\left(c^2\epsilon^{-2}(\eps^2 +1)T \right)}
\end{align*}
which implies \eqref{goal-last} for  $f_t^\eps = N_t^\eps$. \\
Also, it easily follows from \eqref{K-f-bound} and estimates above that $\sup_{\eps \in (0,1)} \E(||Z^\eps||_\alpha) <\infty$, which implies the tightness of $\{Z^\eps\}_{\eps \in (0,1)}$. The tightness of $\{Z^{\eps,h}\}_{\eps \in (0,1)}$ follows similarly.

\end{proof}

\subsection{Upper bound.}
We now prove the LDP upper bound in Theorem \ref{main-thm}. Since $Z^\eps$ is exponentially tight, it suffices to prove:

\begin{proposition}\label{UP}
For all $\varphi \in C_0([0,T], \C)$,
\begin{equation}
\lim_{r \to 0}\limsup\limits_{\epsilon \to 0} \epsilon^2 \log\mathbb P(Z^\epsilon \in B_r(\varphi) )\le-I(\varphi).
\end{equation}
\end{proposition}

For proving the above claim, we will weight probabilities by exponential martingale $M^\eps_{f,g}$ defined by 
\begin{align}\label{our-choice}
M^{\eps}_{f,g}(Z^\epsilon)&= \mathcal E\biggl(\frac{1}{\epsilon}\int_0^T (f_rU_r^\epsilon+ g_rV_r^\epsilon) dB_r\biggr) \nonumber\\ 
&=\exp \biggl(\frac{1}{\epsilon^2}\biggl(\int_0^Tf_r\epsilon U^\epsilon_rdB_r+\int_0^Tg_r\epsilon V^\epsilon_rdB_r\biggr) \nonumber\\& \qquad \qquad -\frac{1}{2\epsilon^2}\int_0^T (f_r^2(U^\epsilon_r)^2+g_r^2(V^\epsilon_r)^2+2f_rg_rU^\epsilon_rV^\epsilon_r)dr \biggr), 
\end{align}
where $f,g \in C^1([0,T],\R)$. Note that we will need to have martingale $M_{f,g}^\eps(Z^\eps)$ to be parametrised by two functions $f,g$. This is owing to the fact that even though $B$ is real valued, $Z^\eps$ is complex valued. Since $Z^\eps$ solves \eqref{Z-eqn}, we have 
\begin{equation}\label{3.7}
d((U^\epsilon_t)^2-(V^\epsilon_t)^2+t)=2\epsilon U^\epsilon_tdB_t, \hspace{2mm} d(U^\epsilon_t V^\epsilon_t)= \eps V_t^\epsilon dB_t.
\end{equation}
Therefore,
\begin{align*}
M^{\eps}_{f,g}(Z^\epsilon) = \exp\biggl(&\frac{1}{2\epsilon^2}\biggl(\int_0^Tf_rd(Re(Z_r^\eps) + r) +\int_0^Tg_rd(Im(Z_r^\eps))\biggr)  \\&- \frac{1}{2\epsilon^2}\int_0^T (f_r^2\frac{|Z_r^\eps| + Re(Z_r^\eps)}{2}+g_r^2 \frac{|Z_r^\eps| - Re(Z_r^\eps)}{2}+f_rg_rIm(Z_r^\eps))dr\biggr).
\end{align*}
Correspondingly, for any $\xi \in C_0([0,T], \C)$, we define  
\begin{align*}
M^{\eps}_{f,g}(\xi) := \exp\biggl(\frac{1}{\epsilon^2}J_{f,g}(\xi)\biggr),
\end{align*}
where 
\begin{align*}
 J_{f,g}(\xi):= &\frac{1}{2}\biggl(\int_0^Tf_rd(Re(\xi_r) + r) +\int_0^Tg_rd(Im(\xi_r))\biggr)  \\&\qquad - \frac{1}{2}\int_0^T (f_r^2\frac{|\xi_r| + Re(\xi_r)}{2}+g_r^2 \frac{|\xi_r| - Re(\xi_r)}{2}+f_rg_rIm(\xi_r))dr. 
\end{align*}

Note that, since $f,g \in C^1([0,T],\R)$, the first two integrals appearing above is well defined for any continuous $\xi$ as a Riemann-Stieltjes integral\footnote{For continuous functions $X,Y$, the Riemann Stieltjes integral $\int X_rdY_r$ is well defined if either of $X$ or $Y$ is of bounded variation.}. Furthermore, using integration by parts formula, 
\[\int_0^Tf_rd(Re(\xi_r) + r) = f_T (Re(\xi_T) + T)-\int_0^T(Re(\xi_r) + r)df_r,\]
and
 \[ \int_0^Tg_rd(Im(\xi_r)) = g_TIm(\xi_T)-\int_0^T Im(\xi_r)dg_r.\]
 
 Therefore, for each fixed $f,g \in C^1([0,T], \R)$, the function $\xi \mapsto J_{f,g}(\xi)$ is continuous on $C_0([0,T], \C)$. We further claim that:
 
 \begin{proposition}\label{J=I}
 For each $\varphi \in C_0([0,T], \C)$, 
 \begin{equation}\label{JI-eqn}
 \sup_{f,g \in C^1([0,T], \R)} J_{f,g}(\varphi) = I(\varphi).
 \end{equation}

 \end{proposition}

The proof of Proposition \ref{J=I} is postponed till section \ref{JIProof}. As a result of this, we have:
 
\begin{proof}[Proof of Proposition \ref{UP}] 
 
 Since $M^{\eps}_{f,g}(Z^\epsilon)$ is a positive local martingale, it is a supermartingale. Hence, $\E(M^{\eps}_{f,g}(Z^\epsilon)) \leq 1$. This implies that 
 \begin{align*}
\mathbb P(Z^\epsilon \in B_r(\varphi)) &= \mathbb E \left(1_{B_r(\varphi)}(Z^\epsilon)\frac{M^{\eps}_{f,g}(Z^\epsilon)}{M^{\eps}_{f,g}(Z^\epsilon)}\right) \\
& \le \sup \limits_{\xi \in B_r(\varphi)} (M^\epsilon_{f,g}(\xi))^{-1}\E(M^{\eps}_{f,g}(Z^\epsilon))\\
&\le \sup \limits_{\xi \in B_r(\varphi)} (M^\epsilon_{f,g}(\xi))^{-1}.
\end{align*}
This implies, using the continuity of $\xi \mapsto M^{\eps}_{f,g}(\xi)$, 
\begin{equation}
\lim_{r \to 0}\limsup\limits_{\epsilon \to 0} \epsilon^2 \log\mathbb P(Z^\epsilon \in B_r(\varphi) )\le -J_{f,g}(\varphi). 
\end{equation}
Minimizing the right hand side over $f,g$ and using Proposition \ref{J=I} completes the proof. 
\end{proof}

\subsection{Some analytical lemmas.}
The proof of Proposition \ref{J=I} requires following optimisation results. The following lemma is well known and it is a consequence of Riesz theorem, see [\cite{YorLDP}, Proposition $3.2$] for details.

\begin{lemma}\label{L2-lemma}
Let $\alpha, \beta \in C_0([0,T], \R)$ such that $\beta $ is non-negative. Assume that 
\begin{equation}\label{alpha-beta} \sup_{f\in C^1([0,T], \R)} \biggl \{\int_0^Tf_r d\alpha_r - \frac{1}{2} \int_0^T f_r^2\beta_r dr\biggr\} <\infty.
\end{equation}
Then $\alpha$ is a absolutely continuous function and there exists a measurable function $k:[0,T] \to \R$ such that $\int_0^T k_r^2 \beta_r dr <\infty$ and $\dot{\alpha}_t = k_t\beta_t$ Lebesgue almost everywhere. 

\end{lemma}

%

Besides the above one dimensional optimisation in $f$, we also need a two dimensional optimisation over functions $f,g$:

\begin{lemma}\label{2-var}
Let $u,v:[0,T]\to \R $ are bounded measurable functions and $p,q \in L^2([0,T], \R)$. Then, 
\begin{equation}\label{p-q} \sup_{f,g\in C^1([0,T], \R)} \int_0^T\{f_ru_rp_r + g_rv_rq_r - (f_ru_r + g_rv_r)^2 \}dr<\infty
\end{equation}
if and only if $p=q$ a.e. on the set $\{uv \neq 0\}$.

\end{lemma} 

\begin{proof}
If $p=q$ a.e. on the set $\{uv \neq 0\}$, then for almost every $r$,
\begin{align*}
&f_ru_rp_r + g_rv_rq_r - (f_ru_r + g_rv_r)^2  \\
&= (f_ru_rp_r + g_rv_rq_r - (f_ru_r + g_rv_r)^2)1_{u_rv_r \neq 0} + (f_ru_rp_r + g_rv_rq_r - (f_ru_r + g_rv_r)^2)1_{u_rv_r = 0} \\
& =(p_r(f_ru_r+ g_rv_r) - (f_ru_r + g_rv_r)^2)1_{u_rv_r \neq 0} + (f_ru_rp_r + g_rv_rq_r - f_r^2u_r^2 + g_r^2v_r^2)1_{u_rv_r = 0} \\
&\leq \frac{1}{4}(p_r^21_{u_rv_r \neq 0} + (p_r^2 + q_r^2)1_{u_rv_r = 0}),
\end{align*}
which implies \eqref{p-q}. \\
Conversely, let us now assume \eqref{p-q} holds. 

For constants $L, \eps >0$, considers functions 
\[ x_r := \frac{L(p_r-q_r) + p_r + q_r}{2u_r}1_{|u_r|\wedge|v_r| \geq \eps} + 1_{|u_r|\wedge|v_r| \leq \eps},\]
and 
\[ y_r := \frac{p_r + q_r -L(p_r-q_r)}{2v_r}1_{|u_r|\wedge|v_r| \geq \eps} + 1_{|u_r|\wedge|v_r| \leq \eps}.
\]
Clearly, $x,y \in L^2([0,T], \R)$. Since $C^1([0,T], \R)$ is dense in $L^2([0,T], \R)$, we can pick sequences $f^n, g^n \in C^1([0,T], \R)$ such that $f^n \to x$ and $g^n\to y$ in $L^2([0,T], \R)$. Since $u,v$ are bounded, it follows that 
\[f^nu+ g^nv \to (p+q)1_{|u|\wedge|v| \geq \eps}+ (u+v)1_{|u|\wedge|v| \leq \eps},\]
and 
\[f^nu- g^nv \to L(p-q)1_{|u|\wedge|v| \geq \eps}+ (u-v)1_{|u|\wedge|v| \leq \eps}\]
in $L^2([0,T], \R)$. This in turn implies that 
\begin{align*} 
\frac{1}{2}\int_0^T &(f^n_ru_r - g^n_rv_r)(p_r -q_r)dr \\ &\to \frac{L}{2}\int_0^T (p_r-q_r)^21_{|u_r|\wedge|v_r| \geq \eps}dr +  \frac{1}{2}\int_0^T (u_r-v_r)(p_r-q_r)1_{|u_r|\wedge|v_r| \leq \eps}dr,
\end{align*}
and
\[\frac{1}{2}\int_0^T (f^n_ru_r + g^n_rv_r)(p_r + q_r)dr - \int_0^T(f^n_ru_r + g^n_rv_r)^2dr \to c \]
where $c$ is independent of $L$. Note that sum of left hand sides of above two equations equals the integral appearing in \eqref{p-q}, which is bounded in $f,g$. This implies that $L\int_0^T (p_r-q_r)^21_{|u_r|\wedge|v_r| \geq \eps}dr$ is bounded. Since $L$ is arbitrary, this implies that 
\[ \int_0^T (p_r-q_r)^21_{|u_r|\wedge|v_r| \geq \eps}dr =0.\]
By letting $\eps \to 0+$, it follows using dominated convergence theorem that 
\[ \int_0^T (p_r-q_r)^21_{|u_r|\wedge|v_r| > 0}dr =0,\]
which concludes the proof.

\end{proof}

\subsection{Proof of Proposition \ref{J=I}.} \label{JIProof}

Let us first assume $I(\varphi) <\infty$. Then, $\varphi = \varphi^h$ for some $h\in H_0^1([0,T],\R)$. Let $\sqrt{\varphi_t} = U_t + iV_t$. Since $\varphi$ solves \eqref{h-phi-eqn-2}, we have 
\begin{equation}
d(U_t^2-V_t^2+t)=2U_tdh_t, 
\end{equation}
and 
\begin{equation}
d(U_t V_t)= V_tdh_t.
\end{equation}
Following a simple rewriting, this implies that 
\begin{align*}
J_{f,g}(\varphi) & = \int_0^T (f_r U_r + g_rV_r)dh_r - \frac{1}{2}\int_0^T (f_rU_r+g_r V_r)^2dr \\ 
& = \int_0^T \bigl\{(f_r U_r + g_rV_r)\dot{h}_r - \frac{1}{2} (f_rU_r+g_r V_r)^2 \bigr\}dr  \\
&\leq \frac{1}{2}\int_0^T \dot{h}_r^2 dr = I(\varphi), 
\end{align*}
which implies $\sup_{f,g \in C^1([0,T], \R)} J_{f,g}(\varphi) \leq I(\varphi)<\infty$. Also, note that 
\[J_{0,g}(\varphi) =  \int_0^T g_rV_r\dot{h}_rdr - \frac{1}{2}\int_0^T g_r^2 V_r^2dr = -\frac{1}{2}\int_0^T(\dot{h}_r - g_rV_r)^2dr + I(\varphi). \]
Since $C^1([0,T], \R)$ is dense in $L^2([0,T], \R)$, 
\[\inf_{g\in C^1([0,T], \R)}\int_0^T(\dot{h}_r - g_rV_r)^2dr = \inf_{g\in L^2([0,T], \R)}\int_0^T(\dot{h}_r - g_rV_r)^2dr.\]
Also, since $\dot{h} \in L^2([0,T], \R)$ and $V$ is a strictly increasing positive function, 
\[\inf_{g\in L^2([0,T], \R)}\int_0^T(\dot{h}_r - g_rV_r)^2dr =0.\]
Hence, $\sup_{f,g \in C^1([0,T], \R)} J_{f,g}(\varphi) = I(\varphi).$ \\

Conversely, now assume that $\sup_{f,g \in C^1([0,T], \R)} J_{f,g}(\varphi) < \infty$. This implies that both $\sup_{f \in C^1([0,T], \R)} J_{f,0}(\varphi) < \infty$ and $\sup_{g \in C^1([0,T], \R)} J_{0,g}(\varphi) < \infty$. Using Lemma \ref{L2-lemma}, this implies that $Re(\varphi)$ and $Im(\varphi)$ are absolutely continuous functions. Furthermore, for some measurable functions $k,l:[0,T] \to \R$ such that 
\begin{equation}\label{need} 
\int_0^T k_r^2(|\varphi_r| + Re(\varphi_r))dr +  \int_0^T l_r^2(|\varphi_r| - Re(\varphi_r))dr <\infty,
\end{equation}
we have

\begin{equation}\label{k-1}
Re(\dot{\varphi}_t) + 1 = \frac{1}{2} k_t(|\varphi_t| + Re(\varphi_t)) \textrm{ a.e.}, 
\end{equation}
and 
\begin{equation}\label{l-1} 
Im(\dot{\varphi}_t) = \frac{1}{2} l_t(|\varphi_t| -Re(\varphi_t)) \textrm{ a.e.}. 
\end{equation}

Next, let $u_t + iv_t = \sqrt{\varphi_t}1_{\varphi_t \in \C\setminus [0,\infty)} + \sqrt{|\varphi_t|}1_{\varphi_t \in [0,\infty)}$. Note that $u_t+iv_t$ is a branch square root of $\varphi$. It follows that $|\varphi_t| = u_t^2 + v_t^2, Re(\varphi_t) = u_t^2 - v_t^2$, and $Im(\varphi_t) = 2u_tv_t$. Hence, $J_{f,g}(\varphi)$ can be written as 

\begin{equation} 
J_{f,g}(\varphi)  = \frac{1}{2}\int_0^T\{f_rk_ru_r^2 + g_rl_rv_r^2 - (f_ru_r + g_rv_r)^2 \}dr.
\end{equation}
Using Lemma \ref{2-var}, we obtain that $k_ru_r = l_rv_r$ a.e. on the set $\{uv\neq 0\}$. 
Now, \eqref{k-1},\eqref{l-1}, $Re(\dot{\varphi}_t) + 1 = k_tu_t^2$, $Im(\dot{\varphi}_t)= l_tv_t^2$, which implies that 
\begin{align*}
\varphi_t &= -t + \int_0^t (k_ru_r^2 + il_rv_r^2)dr \\
&= -t + \int_0^t (u_r + iv_r)(k_ru_r1_{u_rv_r \neq 0} + l_rv_r1_{u_r=0, v_r \neq 0} + k_ru_r1_{u_r\neq 0, v_r = 0} )dr.
\end{align*}
Therefore, $\varphi$ solves \eqref{h-phi-eqn} with $A_t = u_t + iv_t$ and 
\[h_t = \frac{1}{2}\int_0^t (k_ru_r1_{u_rv_r \neq 0} + l_rv_r1_{u_r=0, v_r \neq 0} + k_ru_r1_{u_r\neq 0, v_r = 0} )dr.\]
Note that by \eqref{need}, $h\in H_0^1([0,T], \R)$. Hence, using Lemma \ref{unique}, $\varphi_t \in \C\setminus [0,\infty)$ for all $t>0$ and $\varphi = \varphi^h$. Hence $I(\varphi) <\infty$ and \eqref{J=I} follows from the previous case.

\subsection{Lower bound.}
We now prove the LDP lower bound in Theorem \ref{main-thm}. Let $C_0^2([0,T], \R)$ be the space of continuously twice differentiable $h:[0,T]\to \R$ with $h_0 =0$ and $Y := \{ \varphi^h \textrm{ }\bigl |\textrm{ } h\in C_0^2([0,T],\R)\}$. It follows using denseness of $C_0^2([0,T],\R)$ in $H_0^1([0,T],\R)$ and Lemma \ref{cont-BV} that for each $\varphi$ with $I(\varphi) <\infty$, there exists a sequence $\varphi_n \in Y$ such that $\varphi_n \to \varphi$ uniformly and $I(\varphi_n) \to I(\varphi)$. Thus, it suffices to prove the following to obtain the LDP lower bound for $Z^\eps$.

\begin{proposition}\label{LB}
For any $\varphi \in Y = \{ \varphi^h \textrm{ }\bigl |\textrm{ } h\in C_0^2([0,T],\R)\}$,
\begin{equation}\label{wLB}
\lim_{r \to 0}\liminf\limits_{\epsilon \to 0} \epsilon^2 \log\mathbb P(Z^\epsilon \in B_r(\varphi) )\ge-I(\varphi).
\end{equation}
\end{proposition}

The key ingredient in the proof of above claim is the following observation: 

\begin{proposition}\label{req-conv}
Let $h \in H_0^1([0,T], \R)$ and $Z^{\eps,h}, \varphi^h$ be as described in Section \ref{prelim}. Then, as $\eps \to 0+$, 
\begin{equation}\label{key-step}
 Z^{\eps, h} \overset{\P}{\longrightarrow} \varphi^h.
 \end{equation}
\end{proposition}

The proof of Proposition \ref{req-conv} is postponed till next section. As a result of this, we have:

\begin{proof}[Proof of Proposition \ref{LB}]
Let $\varphi = \varphi^h$ for some $h\in C^2([0,T],\R)$. We introduce a change of measure 
$$\frac{d \mathbb Q}{d \mathbb P}=N^\epsilon$$ where 
\begin{equation}\label{our-choice-2}
N^\epsilon=\exp\left(\frac{1}{\epsilon}\int_0^T \dot h_r dB_r - \frac{1}{2\eps^2} \int_0^T \dot{h}_r^2 dr \right).
\end{equation}

By Girsanov theorem, $B _t - h_t/\eps$ is a standard Brownian motion under $\mathbb{Q}$. Also, using integration by parts, 
\[ \int_0^T \dot h_r dB_r = \dot h_T B_T - \int_0^T B_r\ddot h dr \leq C||B||_\infty\] 
for some constant $C$ depending only on $h$. Therefore, 

\begin{align*}
\mathbb P(Z^\epsilon \in B_r(\varphi)) = \mathbb E \left(1_{B_r(\varphi)}(Z^\epsilon)\frac{N^\epsilon}{N^\epsilon} \right)  &= \mathbb E^\Q \left(1_{B_r(\varphi)}(Z^\epsilon)\exp\left(-\frac{1}{\epsilon}\int_0^T \dot h_rdB_r+\frac{1}{2\epsilon^2}\int_0^T {\dot h}_r^2dr\right) \right)\\
& = \mathbb E\left(1_{B_r(\varphi)}(Z^{\epsilon,h}) \exp\left(-\frac{1}{\epsilon}\int_0^T \dot h_rdB_r -\frac{1}{2\epsilon^2}\int_0^T {\dot h}_r^2dr\right) \right) \\
& \geq \mathbb P\left( Z^{\eps,h} \in B_r(\varphi), ||B||_\infty\leq 1 \right) e^{-\frac{C}{\epsilon}}\exp\left(-\frac{1}{2\epsilon^2}\int_0^T {\dot h}_r^2dr\right).
\end{align*}
Using Proposition \ref{req-conv}, $\mathbb P\left( Z^{\eps,h} \in B_r(\varphi), ||B||_\infty\leq 1 \right)  \to \mathbb P\left(||B||_\infty\leq 1 \right)  >0$. Therefore,
\begin{align}
\lim_{r \to 0}\liminf\limits_{\epsilon \to 0} \epsilon^2 \log\mathbb P(Z^\epsilon \in B_r(\varphi))\ge-I(\varphi).
\end{align}

\end{proof}

\subsection{Proof of Proposition \ref{req-conv}.}

Using Lemma \ref{U-V-bound}, it can be easily seen that as $\eps \to 0+$
\[ \eps \int_0^{\cdot} \sqrt{Z_r^{\eps,h}}dB_r \overset{\P}{\longrightarrow} 0.\]
Since $Z^{\eps, h}$ solve \eqref{Z-h-eqn}, we get that 
\begin{equation}\label{Z-to-0}
Z_t ^{\eps, h} + t + 2 \int_0^t  \sqrt{Z_r^{\eps,h}}dh_r  \overset{\P}{\longrightarrow} 0.
\end{equation}

Now, let $\eps_n \to 0+$ be any sequence. Let us write $Z^n_t =  Z^{\eps_n, h}$. Using the tightness of $Z^{\eps, h}$ (Proposition \ref{p2}), we get that along a subsequence $\eps_{n_k}$, $Z^{n_k} \overset{d}{\to} \varphi,$ where $\varphi$ is some $C_0([0,T], \C)$-valued random variable. Using Skorokhod's representation theorem, there exists $C_0([0,T], \C)$-valued random variables $Y^k$ and $\Psi$ such that $Y^k \overset{d}{=} Z^{n_k} $, $\Psi \overset{d}{=} \varphi$, and $Y^k \to \Psi$ almost surely. Clearly, \eqref{Z-to-0} implies that 
\begin{equation}\label{Y-to-0}
Y_t^k + t + 2 \int_0^t  \sqrt{Y_r^k}dh_r  \overset{\P}{\longrightarrow} 0.
\end{equation}
Next, using Lemma \ref{bsr-exist}, possibly along a subsequence, $\sqrt{Y^k}$ converges uniformly to a branch square root $A_t = A_t(\Psi)$. Therefore, it follows by taking $k\to \infty$ in the above that 
\[ \Psi_t + t + 2 \int_0^t  A_r dh_r  =0 \textrm{ a.s.}.\]
Using Lemma \ref{unique}, this implies that $\Psi = \varphi^h \textrm{ a.s.}$. Hence, $\varphi =  \varphi^h \textrm{ a.s.}. $ Since $\varphi^h$ is deterministic, it follows that $Z^{n_k} \overset{\P}{\longrightarrow} \varphi^h$. Since the limiting object $\varphi^h$ is the same for any sequence $\eps_n\to 0+$, the \eqref{key-step} follows. 

\begin{remark} \label{good-remark}
The Proposition \ref{req-conv} is similar in spirit to continuity of Loewner traces with respect to perturbations in the driving function. This in general is a delicate and difficult problem. However, since we only need convergence in probability in \eqref{key-step}, we get around this difficulty by relying on the uniqueness of solution to \eqref{h-phi-eqn}.
\end{remark}

\bibliographystyle{alpha}
\bibliographystyle{acm}	
\bibliography {LDP}

\end{document}